\documentclass{amsart}
\usepackage[mathscr]{eucal}
\usepackage{color}
\usepackage{xypic}
\usepackage{amsfonts}   
\usepackage{amsmath}
\usepackage{amsthm}
\usepackage{amssymb}
\usepackage{latexsym}
\newtheorem{theorem}[]{Theorem}
\newtheorem{proposition}[]{Proposition}
\newtheorem{corollary}[]{Corollary}
\newtheorem{lemma}[]{Lemma}
\theoremstyle{definition}
\newtheorem{definition}[]{Definition}
\newtheorem{remark}[]{Remark}

\let\nc\newcommand

\nc{\la}{\label}

\def\bthm{\begin{theorem}}
\def\ethm{\end{theorem}}
\def\blemma{\begin{lemma}}
\def\elemma{\end{lemma}}
\def\bproof{\begin{proof}}
\def\eproof{\end{proof}}
\def\bprop{\begin{proposition}}
\def\eprop{\end{proposition}}
\def\bcor{\begin{corollary}}
\def\ecor{\end{corollary}}

\def\Z{\mathbb{Z}}

\def\k{\mathsf k}

\nc{\Hom}{{\rm{Hom}}}
\nc{\chara}{{\rm{char}}}
\nc{\Ext}{{\rm{Ext}}}
\nc{\HOM}{\underline{\rm{Hom}}}
\nc{\EXT}{\underline{\rm{Ext}}}
\nc{\TOR}{\underline{\rm{Tor}}}
\nc{\End}{{\rm{End}}}
\nc{\GL}{{\rm{GL}}}
\nc{\SL}{{\rm{SL}}}
\nc{\Rep}{{\rm{Rep}}}
\nc{\ad}{{\rm{ad}}}
\nc{\dlim}{\varinjlim}

\newcommand{\Frac}{{\rm{Frac}}}

\begin{document}
\title[Quantum linear Galois orders]{Quantum linear Galois orders} 
%Galois orders and rings of fractions} 

\author{Vyacheslav Futorny}
\author{Jo\~ao Schwarz}
\address{Instituto de Matem\'atica e Estat\'istica, Universidade de S\~ao
Paulo,  S\~ao Paulo SP, Brasil} \email{futorny@ime.usp.br,}\email{jfschwarz.0791@gmail.com}

\begin{abstract}
We define a class of quantum linear Galois algebras which include    
 the universal enveloping algebra $U_q(gl_n)$, the quantum Heisenberg Lie algebra and other quantum orthogonal Gelfand-Zetlin algebras of type $A$,    the subalgebras of $G$-invariants of the quantum affine space, quantum torus  for $G=G(m, p, n)$, and of the quantum Weyl algebra for $G=S_n$. We show that all quantum linear Galois algebras satisfy the quantum Gelfand-Kirillov conjecture.  Moreover, it is shown that the 
 the subalgebras of  invariants of the quantum affine space and of quantum torus  for the reflection groups and of the quantum Weyl algebra for symmetric groups  are, in fact, Galois orders 
 over an adequate commutative subalgebras and free as  right (left) modules over these subalgebras. In the rank $1$ cases the results hold for an arbitrary finite group of automorphisms when the field is $\mathbb C$. 
 \end{abstract}

 \maketitle

\section{Introduction}

The purpose of this paper is to quantize the results of \cite{FS2}, where subalgebras of invariants of Weyl algebras were studied for irreducible reflection groups. It was shown that in many cases these subalgebras have a structure of Galois orders over certain commutative domains. This feature indicates  a hidden skew group algebra structure of all these algebras.  

The theory of Galois rings and orders developed in  \cite{FO1}, \cite{FO2}. Classical examples includes   finite $W$-algebras of type $A$   \cite{FMO}, in particular the universal enveloping algebra of $gl_n$, and generalized Weyl algebras of rank $1$ over integral domains with infinite order automorphisms  \cite{Bavula}. The importance of the Galois order structure is in their representation theory, where one can effectively study the Gelfand-Tsetlin categories of modules with torsion for certain maximal commutative subalgebras \cite{Ovsienko}, \cite{FO2}.
 
Our  main objects of interest are the following  quantum algebras: the quantum affine space $O_q(\k^{2n})$, the quantum torus $O_q(\k*^{2n})$ and
 the quantum Weyl algebra $A_n^q(\k)$.

Our first  result shows that the subring of invariants $O_q(\k^{2n})^G$ of the quantum affine space is a Galois order over certain polynomial subalgebra when $G=G_m^{\otimes n}$ is a product of  cyclic groups (Proposition \ref{prop-affine-cyclic-order}) or $G=G(m,p,n)$ is one of non exceptional reflection groups (Theorem \ref{thm-affine-refl-order}):

\begin{theorem}\label{thm-main1} 
If  $G$ is a product of $n$ copies of a cyclic group of fixed finite order or one of the irreducible non exceptional reflection groups $G(m,p,n)$, then the invariant subring $O_q(\k^{2n})^G$ of the quantum affine space is 
 is a Galois order over  a polynomial  subalgebra $\Gamma$ of $O_q(\k^{2n})^G$. Moreover, $O_q(\k^{2n})^G$ is free as a left (right) $\Gamma$-module.
\end{theorem}

Theorem \ref{thm-main1} can be easily generalized to the case of the quantum torus (Theorem \ref{thm-torus-refl-order}):

\begin{theorem}\label{thm-main2} 
For every $G=G(m,p,n)$ the invariant subring $O_q(\k*^{2n})^G$ of the quantum torus is 
 is a Galois order over $\Gamma=\k[x^{\pm 1}_1, \ldots, x^{\pm 1}_n]^G$ in  $\k(x_1,\ldots, x_n)*\mathbb{Z}^n)^G$. Moreover,   $O_q(\k^{2n})^G$ is free as a left (right) $\Gamma$-module.
\end{theorem}

We have the following generalization of Theorem \ref{thm-main1} for quantum planes (Theorem \ref{thm-plane-order}) and the first quantum Weyl algebra (Proposition \ref{prop-A1-order}) when $\k=\mathbb C$:

\begin{theorem}\label{thm-main3} Let $A\in \{O_q(\mathbb C^{2}),   A_1^q(\mathbb C)\}$.
For every finite group $G$ of automorphisms of  $A$, the subring of invariants  $A^G$ is a Galois order over 
a certain polynomial subalgebra  $\Gamma$ in one variable. Moreover,  $A^G$  is free as a left (right) $\Gamma$-module.
\end{theorem}

It was shown in \cite{FS1} that  $A_n(\k)^{S_n}$ is a Galois order over some polynomial algebra.  We prove the quantum analog of this result for $A_n^{{q}}(\k)^{S_n}$ (Theorem \ref{thm-quantum-order}). 

In Section \ref{QGKC} we address the quantum Gelfand-Kirillov conjecture for various algebras.  We introduce  a class of quantum linear Galois algebras and show that 
the quantum Gelfand-Kirillov conjecture is valid in this class (Theorem \ref{thm-qlga}). Quantum linear Galois algebras include the quantum Orthogonal Gelfand-Zetlin algebras of type $A$ 
(in particular, the universal enveloping algebra $U_q(gl_n)$ and the quantum Heisenberg Lie algebra),    $O_q(\k^{2n})^G$ and $O_q(\k*^{2n})^G$ for $G=G(m,p,n)$, 
 $A_n^{{q}}(\k)^{S_n}$. When $n=1$ the group $G$ in all cases can be arbitrary. 
 
 We also compute the skew fields of fractions for the quantum $2$-sphere and for the quantum group $O_{q^2}(so(3,\mathbb{C}))$. Finally, we show that the subalgebra of $G_m$-invariants of$U(sl_2)$ for the cyclic group $G_m$ of order $m$ is birationally equivalent to $U(sl_2)$ in spite of the rigidity of the latter.

\

\noindent{\bf Acknowledgements.}  V.F. is
supported in part by  CNPq grant (200783/2018-1) and by 
Fapesp grant (2014/09310-5).  J.S. is supported in part by Fapesp grants (2014/25612-1)  and  (2016/14648-0).

\section{Preliminaries}
All rings and fields in the paper are assumed to be $\k$-algebras over an algebraically closed field $\k$ of characteristic $0$.

For $q\in \k$ we denote by $\k_q[x,y]$  the  \emph{quantum plane}  over $\k$ is defined as $\k\langle x,y\mid
yx=qxy\rangle$.    In this paper we will always assume that $q$ is not a root of unity.
%Let $\bar q=(q_1,\ldots,q_n)\in(\k\backslash\{0\})^n$. 
Let $\overline{q}=(q_1, \ldots, q_n) \in \k^n$ be an $n$-tuple whose components are  non zero  and  non roots of unity.
 The tensor product of quantum planes $k_{q_1}[x_1,y_1] \otimes \ldots \otimes \k_{q_n}[x_n,y_n]$ will be called \emph{quantum affine space} and will be denoted by $O_{\overline{q}}(\k^{2n})$. If  $q_1= \ldots = q_n=q$, we will use the notation $O_q(\k^{2n})$.

Denote by $A_1^q(\k)$
 the first quantum Weyl algebra  defined as  $\k\langle x,y\mid
yx-qxy=1\rangle$ and set $$A_n^{\overline{q}}(\k)=A_1^{q_1}(\k)\otimes_\k \cdots \otimes_\k A_1^{q_n}(\k)$$ for any positive integer $n$.
 Again, if $q_1= \ldots = q_n=q$ then we simply denote it by $A_n^{{q}}(\k)$.

The quantum affine space $O_{\overline{q}}(\k^{2n})$ and the quantum Weyl algebra $A_n^{\overline{q}}(\k)$ are birationally equivalent, that is they have isomorphic skew fields of fractions
\cite{BG}.

\subsection{Galois orders}
We recall the concepts of Galois rings and Galois orders from \cite{FO1}. Let $\Gamma$ be a commutative domain and $K$  the field of fractions of $\Gamma$.
Let  $L$ be a finite Galois extension of  $K$ with the Galois group $G=Gal(L,K)$, $\mathfrak{M}\subset Aut_\k \, L $  a monoid 
satisfying the following condition: if $m,m' \in \mathfrak{M}$ and their restrictions to $K$ coincide, then $m=m'$. Consider the action of $G$ on $\mathfrak{M}$ by conjugation.

A finitely generated $\Gamma$-ring $U$ in $(L*\mathfrak{M})^G$ is called a \emph{Galois ring over $\Gamma$} if $KU=KU=(L*\mathfrak{M})^G$.
A Galois ring over $\Gamma$ is called a \emph{right (left) Galois order over $\Gamma$} if for every right (left) finite dimensional $K$-vector subspace $W \subset \mathfrak{K}$, $W \cap \Gamma$ is a finitely generated right (left) $\Gamma$-module.
If $U$  is both left and right Galois order over $\Gamma$, then we say that $U$ is a \emph{Galois order over $\Gamma$}.

If  $x = \sum_{m \in \mathfrak{M}} x_m m\in L*\mathfrak{M}$ then set $$supp \, x =\{m\in \mathfrak{M} | , x_m\neq 0 \}.$$

We have 

\begin{proposition} \cite{FO1} \label{prop-supp}
Let $\Gamma\subset U$ be a commutative domain and  $U\subset (L*\mathfrak{M})^G$.
\begin{itemize}
\item[(i)]
If $U$ is generated by $u_1,\ldots, u_k$ as a $\Gamma$-ring and
 $\bigcup_{i=1}^k supp \, u_i$ generates $\mathfrak{M}$ as a monoid, then $U$ is a Galois ring  over $\Gamma$.
\item[(ii)]
Let $U$ be a Galois ring over $\Gamma$ and $S=\Gamma\setminus \{0\}$. Then $S$ is a left and right Ore  set, and the localization of $U$ by $S$ both on the left and on the right is isomorphic to $(L*\mathfrak{M})^G$.
\end{itemize}
\end{proposition}

We also recall the following characterization of Galois orders.

\begin{proposition} \cite{FO1} \label{prop-proj}
Let $\Gamma$ be a commutative Noetherian domain with the field of fractions $K$. 
If $U$ is a Galois ring over $\Gamma$ and $U$ is a left (right) projective $\Gamma$-module, then $U$ is a left (right) Galois order over $\Gamma$.
%If $U$ is a Galois order over  $\Gamma$  then $\Gamma$ is a Harish-Chandra subalgebra.????????????????????????
\end{proposition}

%\begin{proposition} \cite{Futorny1}\label{prop-HC}
%If $U$ is a Galois Order over a finitely generated Noetherian $\Gamma$, then the subalgebra is Harish-Chandra.
%\end{proposition}

\begin{remark}\label{rem-Ore}
Let $D$ be a  commutative domain, finitely generated as a $k$-algebra, $\sigma \in Aut_\k \, D$ and $A=D[x; \sigma]$  the skew
    polynomial  Ore extension, where  $x d=\sigma(d)x$, for all $d\in
    D$. Then $D[x; \sigma]\simeq D* \mathcal M$, where  $$\mathcal M =\{\sigma^n\mid n=0, 1,
    \ldots\}    \simeq \mathbb N.$$
The isomorphism is identity on $D$ and sends $x$ to  the generator $\overline{1}$ of the monoid $\mathbb N$ and  $\overline{1}$ acts on $D$ as $\sigma$. 
Then for $L=K$, the field of fractions of $D$ and for $ G=\{e\}$ we have that the algebra $A$ is a Galois
    ring (order) over $D$ in $K*\mathcal M$.
The localization of $A$ by $x$ is isomorphic to $D*\mathbb{Z}$. 
\end{remark}

\subsection{Invariant subalgebras}

We will use the following two results on the subalgebras of invariants in the non commutative setting. The first is the result of Montgomery and Small which generalizes the Hilbert-Noether theorem.

\begin{theorem}\label{Montgomery}
Let $A$ be a commutative Noetherian ring, and $R \supset A$ an overring such that $A$ is central and $R$ is a finitely generated $A$-algebra. Let $G$ be a finite group of $A$-algebra automorphisms of $R$ such that $|G|^{-1} \in R$.  If $R$ is left and right Noetherian then $R^G$ is a finitely generated $A$-algebra.
\end{theorem}

The following connects the projectivity of subalgebras of invariants with the projectivity of the algebra itself as modules over respective commutative subalgebras. 

\begin{lemma} \label{lemma2}\cite{FS2} Let $U$ be an associative algebra and $\Gamma\subset U$ a Noetherian commutative subalgebra.   Let $H$ be a finite group of automorphisms of $U$ such that $H(\Gamma)\subset \Gamma$.
If $U$ is projective right (left) $\Gamma$-module and $\Gamma$ is projective over 
 $\Gamma^H$, then $U^H$ is
projective right (left) $\Gamma^H$-module.
\end{lemma}

\subsection{Generalized Weyl algebras}

We will often use a realization of a given algebra as a \emph{generalized Weyl algebra} \cite{Bavula}. 
Let $D$ be a ring, $\sigma=(\sigma_1,\ldots, \sigma_n)$ an $n$-tuple of commuting automorphisms of $D$,  $a=(a_1,\ldots, a_n)$  nonzero elements of the center of $D$ and $\sigma_i(a_j)=a_j, j \neq i$. The generalized Weyl algebra $D(a, \sigma)$  is generated over $D$ by $X_i, Y_i$, $i=1,\ldots, n$ subject to the  relations: 

\[ X_i d = \sigma_i (d) X_i; \, Y_i d= \sigma_i^{-1}(d) Y_i, \, d \in D, i=1, \ldots , n , \]
\[ Y_i X_i = a_i; \, X_i Y_i = \sigma_i(a_i), \, i=1 ,\ldots , n \, ,\]
\[ [Y_i, X_j]=[Y_i, Y_j]=[X_i, X_j]=0 \, , i \neq j.\]

We will assume that $D$ is a Noetherian domain which is  finitely generated  $\k$-algebra. Fix a basis $e_1, \ldots, e_n$ of the free abelian group $\mathbb{Z}^n$. 
There is natural embedding of $D(a, \sigma)$ int the skew group ring $D*\mathbb{Z}^n$, where the  action on $D$ is defined as follows: $re_i$ acts as $\sigma_i^r$, for all $i$ and $r \in \mathbb{Z}$. Moreover, this embedding is an  isomorphism if each $a_i$ is a unit in $D$, $i=1, \ldots, n$ (cf. \cite{FS2}, Proposition 4).  Both algebras  algebras, $D(a, \sigma)$ and $D*\mathbb{Z}^n$, admit the skew fields of fractions.  Hence, following the discussion above we have

\begin{proposition}\label{prop-birational}
The algebras $D(a, \sigma)$ and $D*\mathbb{Z}^n$ have isomorphic skew fields of fractions.
\end{proposition}

Note that, if  $\sigma_1,\ldots, \sigma_n$ are linearly independent over $\Z$, then $D(a, \sigma)$ is a Galois order over $D$ in the skew group ring $(\Frac \ D)*\mathbb{Z}^n$ (cf. \cite{FS2}, Theorem 5).

\section{Invariants of quantum affine spaces}
% and the quantized Weyl algebra}
In this section we consider the invariants of quantum affine space $O_q(\k^{2n})$. Fix any integer $m>1$ and let $G_m\subset \k$ be a  cyclic group of order $m$. Our first group 
 $G=G_m^{\otimes n}$ is the product of $n$ copies of $G_m$.
Consider the following natural action of
$G_m^{\otimes n}$  on $O_q(\k^{2n})$: if  $g=(g_1, \ldots, g_n)\in G$ then $g(x_i)=g_ix_i$, $g(y_i)=y_i$, $i=1, \ldots, n$. This action was defined in \cite{Hartwig}, however we are using the defining relations as in \cite{Dumas}.

We have

\begin{proposition}\label{prop-affine-cyclic}
The invariant subspace $O_q(\k^{2n})^{G_m^{\otimes n}}$ is isomorphic to $O_{q^m}(\k^{2n})$.
\begin{proof}
The isomorphism just sends $x_i$ to $x_i^m$ and $y_i $ to $y_i$, $i=1, \ldots, n$.
\end{proof}
\end{proposition}

Consider the free monoid $\mathbb{N}^n$ with generators $\epsilon_1, \ldots, \epsilon_n$ and the skew monoid ring $\k[x_1,\ldots,x_n]*\mathbb{N}^n$, where
 $\mathbb{N}^n$ acts as follows: $\epsilon_i(x_i) = qx_i$,  $\epsilon_i(x_j) = x_j$,  $j \neq i$,  $i, j=1, \ldots, n$.

\begin{proposition}\label{prop-affine-cyclic-order}
Quantum affine space $O_q(\k^{2n})$ is isomorphic to $\k[x_1,\ldots,x_n]*\mathbb{N}^n$. In particular, $O_q(\k^{2n})$ is a Galois ring over $\Gamma=\k[x_1,\ldots,x_n]$ in $\k(x_1,\ldots,x_n)*\mathbb{N}^n$.
\end{proposition}

\begin{proof}
The isomorphism is given by: $x_i \mapsto x_i$, $y_i \mapsto \epsilon_i$, $i=1,\ldots, n$. The rest is clear.
\end{proof}

For $m \geq 1$, $n \geq 1$, $p|m, p>0$ denote by $A(m,p,n)$  the subgroup of $G_m^{\otimes n}$ consisting of elements $(h_1,\ldots, h_n)$ such that $(h_1h_2 \ldots h_n)^{m/p} = id$. 
 The groups $G(m,p,n) = A(m,p,n) \rtimes S_n$ were introduced by Shephard and Todd and describe all irreducible non-exceptional  complex reflection groups. Here 
  $S_n$  acts on  $A(m, p, n)$ by permutations. 
 
Let $G=G(m,p,n)$, and   
 consider the following action of $G$ on $O_q(\k^{2n})$: $h=(g,\pi) \in G$,  $g=(g_1, \ldots, g_n)\in G_m^{\otimes n}, \, \pi \in S_n$, with $h(x_i)= g_i x_{\pi(i)}$, $h(y_i)=y_{\pi(i)}$, $i=1,\ldots, n$.  The  group $G$ also acts on $\k[x_1,\ldots,x_n]* \mathbb{N}^n$: the action on $x_i$ is the same as above, and $h(\epsilon_i)= \epsilon_{\pi(i)}$.
  Clearly,  $G$ acts on $\mathbb{N}^n$  by conjugations, and the isomorphism in  Proposition \ref{prop-affine-cyclic-order} is $G$-equivariant.  Hence, 
  $O_q(\k^{2n})^{G}$ and $(\k[x_1,\ldots,x_n]*\mathbb{N}^n)^{G}$ are canonically isomorphic.  Hence, $O_q(\k^{2n})^G$   
 is a Galois order over $\Gamma=\k[x_1, \ldots, x_n]^{G}$.
  Taking into account that $\Gamma$ is a polynomial algebra and applying 
  Proposition \ref{prop-proj}, Lemma \ref{lemma2} and  \cite{Bass}, Corollary 4.5,
   we have

\begin{theorem}\label{thm-affine-refl-order}
For every $G=G(m,p,n)$ the invariant subring $O_q(\k^{2n})^G$ of the quantum affine space is 
 is a Galois order over $\Gamma=\k[x_1, \ldots, x_n]^{G}$.   Moreover,   $O_q(\k^{2n})^G$ is free as  left (right) $\Gamma$-modules. \end{theorem}

\subsection{Invariants of quantum torus}

One can extend  Theorem \ref{thm-affine-refl-order} to \emph{quantum torus} $O_q(\k*^{2n})^G$, which is the localization of $O_q(\k^{2n})^G\simeq \k[x_1,\ldots, x_n]*\mathbb{N}^n$ by 
$x_1,\ldots, x_n$, $y_1,\ldots, y_n$. 
Hence, 
$$O_q(\k*^{2n})^G\simeq \k[x^{\pm 1}_1,\ldots, x^{\pm 1}_n]*\mathbb{Z}^n.$$  

We also have by Proposition \ref{prop-affine-cyclic}:
 $$O_q(\k*^{2n})^{G_m^{\otimes n}}\simeq O_{q^m}(\k*^{2n}).$$  Using the arguments  before Theorem \ref{thm-affine-refl-order} we immediately obtain

\begin{theorem}\label{thm-torus-refl-order}
For every $G=G(m,p,n)$ the invariant subring $O_q(\k*^{2n})^G$ of the quantum torus is 
 is a Galois order over $\Gamma=\k[x^{\pm 1}_1, \ldots, x^{\pm 1}_n]^{G}$ in  $\k(x_1,\ldots, x_n)*\mathbb{Z}^n)^{G}$. Moreover,   $O_q(\k^{2n})^G$ is free as a left (right) $\Gamma$-module.
\end{theorem}

\subsection{Quantum complex plane}
In this section we assume that 
$\k = \mathbb{C}$.

\begin{proposition}\label{prop-Dumas}
Consider any finite group $G$ of automorphisms of the quantum plane $\mathbb{C}_q[x,y]$.  Then the ring of invariants  $\mathbb{C}_q[x,y]^G$ is embedded into the Ore extension 
$\mathbb{C}_q[x,y]_x^G \cong \mathbb{C}(x^m)[v; \sigma]$, where $\sigma(x^m)=q^n x^m$ for some $n, m >0$ and $v = x^k y^l$, $l, k>0$.
\end{proposition}

\begin{proof} The action of $G$ on the quantum plane $\mathbb{C}_q[x,y]$  extends naturally to its action on the localization of $\mathbb{C}_q[x,y]$ by $x$.
It was shown in \cite{Alev4} that every  finite group $G$ of automorphisms of  the quantum plane is a subgroup of the torus $\mathbb{C}*^2$, and thus has the form $G_m \times G_{m'}$ for cyclic groups of orders $m$ and $n$ respectively. 
Let $g'$ be a generator of $G_m$ and $g''$ a generator of $G_{m'}$. Then $(g'^k,g''^l)(x)=\alpha^k x, \, (g'^k,g''^l)(y)=\beta^ly$, where $\alpha$ is a primitive $m$-th root of unity, and $\beta$ is a primitive $m''$-th root of unity.
 The subring of $G$-invariants of the localized ring $\mathbb{C}_q[x,y]_x$ is the Ore extension $\mathbb{C}(x^m)[v; \sigma]$,  where $\sigma(x^m)=q^n x^m$ for some $n$ and $m$ by \cite{Dumas}, 3.3.3. Multiplying $v$ by $x^m$ sufficiently many times, we can assume it to be  in the claimed form.
\end{proof}

We  have the following general result about the invariants of the quantum plane.

\begin{theorem}\label{thm-plane-order}
For every finite group $G$ of automorphisms of  the quantum plane $\mathbb{C}_q[x,y]$ the subring of invariants  $\mathbb{C}_q[x,y]^G$ is a Galois order over 
a certain polynomial subalgebra  $\Gamma$. Moreover,  $\mathbb{C}_q[x,y]^G$  is free as a left (right) $\Gamma$-module.
\end{theorem}

\begin{proof}
The subring of invariants  $\mathbb{C}_q[x,y]^G$ $\mathbb{C}_q[x,y]^G$ is embedded into $\mathbb{C}(x^m)[v;\sigma] \cong \mathbb{C}(x^m)*\mathbb{N}$ by Proposition \ref{prop-Dumas}, where the generator $\overline{1}$ of $\mathbb{N}$ acts as follows: $\overline{1}(x^m)=q^n x^m$. Also, $v = x^k y^l$ is $G$-invariant and it is mapped to $\overline{1}$ under the isomorphism above. We conclude that 
$\mathbb{C}_q[x,y]^G$ is a Galois order over  $\mathbb{C}[x^m]$ (cf. Remark \ref{rem-Ore}). The rest follows   from  Proposition \ref{prop-proj}, Lemma \ref{lemma2} and \cite{Bass}, Corollary 4.5. 
\end{proof}

\section{Invariants of quantum Weyl algebras}

Consider now the first quantum Weyl algebra $A_1^q(\k)$, generated over $\k$ by $x$ and $y$ subject to  the relation $yx-qxy=1$. It can be realized as a generalized Weyl algebra
 $D(a, \sigma)$ with  $D=\k[h]$, $a=h$, $\sigma(h)=q^{-1} (h-1)$ and generators $X, Y$. The isomorphism is given as follows: $yx \mapsto h$, $x \mapsto X$, $y \mapsto Y$.
Then $A_1^q(\k)$ is a Galois order over $D$ by \cite{FO1}, as $q$ is not root of unity and $\sigma$ has an infinite order. Moreover, the quantum Weyl algebra $A_n^{{q}}(\k)\simeq
 A_1^q(\k)^{\otimes n}$ is a Galois order over $\Gamma=\k [h_1,\ldots, h_n]$  in $\k(h_1,\ldots, h_n)*\mathbb{Z}^n$, where a basis $\epsilon_1, \ldots, \epsilon_n$ of $\mathbb{Z}^n$ atcs on $\Gamma$ as expected: $\epsilon_i(h_i)=q^{-1}(h_i-1)$; $\epsilon_i(h_j)=h_j$, $i,j=1,\ldots, n$. The embedding is given by:

\[ y_i x_i \mapsto h_i, x_i \mapsto \epsilon_i, \, y_i \mapsto h_i^{-1} \epsilon_i^{-1},\]
$i=1,\ldots, n$

Consider the subring of invariants $A_n^{{q}}(\k)^{S_n}$, where $S_n$ acts by simultaneous permutations of the variables $y_i$  and $x_i$, $i=1,\ldots, n$. 
Using the structure of the quantum Weyl algebra $A_n^{{q}}(\k)$ as a Galois order over $\Gamma=\k [h_1,\ldots, h_n]$  in $\k(h_1,\ldots, h_n)*\mathbb{Z}^n$
we  obtain an embedding of $A_n^{{q}}(\k)^{S_n}$ into the ring $(\k (h_1,\ldots, h_n)*\mathbb{Z}^n)^{S_n}$, where $S_n$ permutes  $h_1,\ldots,h_n$ and acts on $\mathbb{N}^n$ by conjugation: if $\pi \in S_n$ then $\pi(\sigma_i)=\sigma_{\pi(i)}$.

\begin{theorem}\label{thm-quantum-order}
$A_n^{{q}}(\k)^{S_n}$ is a Galois order over $\Gamma=\k[h_1,\ldots, h_n]^{S_n}$. Moreover, $A_n^{{q}}(\k)^{S_n}$ is free as a left (right) $\Gamma$-module.
\end{theorem}

\begin{proof}
The algebra $A_n^{{q}}(\k)^{S_n}$  is finitely generated  by Theorem \ref{Montgomery}. Choose  generators $u_1, \ldots, u_k$ and add to this list the  elements $x_1 + \ldots + x_n$ and $y_1+ \ldots + y_n$. 
 The images of the latter two elements in $(\k(h_1,\ldots, h_n)*\mathbb{Z}^n)$ are $\epsilon_1 + \ldots + \epsilon_n$ and $h_1^{-1} \epsilon_1^{-1} + \ldots + h_n^{-1} \epsilon_n^{-1}$ respectively. Hence the support of their image generate $\mathbb{Z}^n$ as a group, and the first statement follows from Proposition \ref{prop-supp}.  The seond statement follows  from  Proposition \ref{prop-proj}, Lemma \ref{lemma2} and \cite{Bass}, Corollary 4.5. 
\end{proof}

We have the following analog of Theorem \ref{thm-plane-order} for the first quantum Weyl algebra when $\k = \mathbb{C}$:

\begin{proposition}\label{prop-A1-order}
Let $G$ be any finite group  of automorphisms of $A_1^q(\mathbb{C})$. Then the invariant subring $A_1^q(\mathbb{C})^G$ is a Galois order over  $\Gamma=\mathbb C[x^m]$ in $\mathbb C(x^m)*\mathbb{N}$. Moreover, $A_1^q(\mathbb{C})^G$ is free as a left (right) $\Gamma$-module.
\end{proposition}

\begin{proof}
Again, by  Alev and Dumas (\cite{Alev2}), every finite group $G$ of automorphisms of $A_1^q(\mathbb{C})$ is of the form $G_m$, where the generator of $G_m$  acts by: $x \mapsto \alpha x$, $y \mapsto \alpha^{-1} y$ for some  $m$th primitive root of unity $\alpha$.  Localization of $A_1^q(\mathbb{C})$ by $x$ is isomorphic to  $\mathbb{C}(x)[z, \sigma]$, 
with $z=(q-1)xy +1$ and $\sigma(x)=qx$. On the other hand, 
 $\mathbb{C}(x)[z, \sigma]$ is just the localization of $\mathbb{C}_q [x, z]$ by $x$. By Theorem \ref{thm-plane-order} we obtain an embedding of $A_1^q(\mathbb{C})^G$ into $\mathbb{C}(x^m)[v; \sigma] \cong \k(x^m)*\mathbb{N}$, where $\sigma (x^m)=q^n x^m$. 
\end{proof}

\section{Quantum Gelfand-Kirillov conjecture}\label{QGKC}

The \emph{quantum Gelfand-Kirillov conjecture} (cf. \cite{BG}, \cite{FH})
compares the skew field of fractions of a given algebra with \emph{quantum Weyl fields}, that is the skew field of fractions of the tensor product of quantum Weyl algebras
$A_1^{q_1}(\k)\otimes_\k \cdots \otimes_\k A_1^{q_n}(\k)$ (or, equivalently, of some quantum affine space).    An algebra $A$ is said to satisfy the  quantum Gelfand-Kirillov conjecture if $\Frac (A)$ is isomorphic to 
a quantum Weyl field over a purely transcendental extension of $\k$.  We will say that two domains $D_{1}$ and $D_{2}$ are \emph{birationally
equivalent} if $\Frac (D_{1})\simeq \Frac (D_{2})$.

The quantum Gelfand-Kirillov conjecture is strongly connected with the \\ \emph{$q$-difference Noether} \emph{problem} for reflection groups introduced in \cite{Hartwig}. This problem asks whether the invariant quantum Weyl subfield 
$(\Frac A_n^{{q}}(\k))^W$ is isomorphic to some quantum Weyl field, where $W$ is a reflection group. The
 positive solution of the $q$-difference Noether problem was obtained in \cite{Hartwig} for classical reflection groups. Using this fact, the validity of the  quantum Gelfand-Kirillov conjecture was shown for the quantum universal enveloping algebra $U_q(gl_n)$ (\cite{FH}) and for the quantum Orthogonal Gelfand-Zetlin algebras of type $A$ (\cite{Hartwig}). The latter class includes  the simplyconnected quantized form of $gl_n$, $\check{U}(gl_n)$ and the quantized Heisenberg Liealgebra among the others.

\subsection{Functions on the quantum $2$-sphere}

Denote by $A(S^2_\lambda)$ the algebra of functions on the quantum $2$-sphere. The algebra $A(S^2_\lambda)$ is the quotient of $\mathbb{C} \langle X,Y,H \rangle$ by the relations

$$XH = \lambda HX, \ YH=\lambda^{-1} HY, $$
$$ \lambda^{1/2}YX=-(c-H)(d+H), \,  \lambda^{-1/2}XY = -(c- \lambda H)(d+ \lambda H). $$

It can be realized as a generalized Weyl algebra $\mathbb{C}[H](a, \sigma)$, where $$a = -\lambda^{-1/2}XY (c-H)(d+H))$$ and  $\sigma(H) = \lambda H$.
By Proposition \ref{prop-birational}, $\mathbb{C}[H](a, \sigma)$  
 is birationally equivalent to  $\mathbb{C}[H] * \mathbb{Z}$, where $\mathbb{Z}$ is generated by $\overline{1}$ and $\overline{1}(H)= \lambda H$. Applying Proposition \ref{prop-affine-cyclic-order} 
 we obtain that $A(S^2_\lambda)$ is birationally equivalent to the quantum plane with parameter $\lambda$.  Hence,  $A(S^2_\lambda)$ satisfies the quantum Gelfand-Kirillov conjecture, that is

\begin{corollary}\label{cor-sphere}
%\emph{(Quantum Gelfand-Kirillov Conjecture for the quantum $2$-sphere)}
$\Frac \, A(S^2_\lambda) \cong \Frac \, \k_\lambda[x,y]$.
\end{corollary}

\subsection{The quantum group $O_{q^2}(so(3,\mathbb{C}))$}

Let $A=O_{q^2}(so(3,\mathbb{C}))$. The algebra $A$ can be realized as a generalized Weyl algebra  
$\mathbb{C}[H,C](\sigma, a)$, where $a= C+H^2/q(1+q^2))$ and $\sigma(C)=C$, $\sigma(H)=q^2 H$. 
By Proposition \ref{prop-birational}, 
 $A$ is birationally equivalent to $\mathbb{C}[C,H]*\mathbb{Z}$, where $\mathbb{Z}$ is generated by $\overline{1}$ acting  as $\sigma$ on $\mathbb{C}[C,H]$. 
 Since $C$ is invariant by $\sigma$, this ring is clearly birationaly equivalent  to $\mathbb{C}[C] \otimes (\mathbb{C}[H] * \mathbb{Z})$.  Applying Proposition \ref{prop-affine-cyclic-order} 
 we obtain that  $A$ satisfies the quantum Gelfand-Kirillov conjecture, that is

\begin{corollary}\label{cor-so}
%\emph{(Quantum Gelfand-Kirillov Conjecture for the quantized $so(3,\mathbb{C})$)}
$\Frac \, O_{q^2}(so(3,\mathbb{C})) \cong \Frac \, (\mathbb{C}(C) \otimes \mathbb{C}_{q^2}[x,y])$.
\end{corollary}

\subsection{Quantum Linear Galois Algebras}
In this section we obtain 
 a quantum version of the theory of linear Galois algebras developed in \cite{Eshmatov}. The field $\k$ is assumed to be the field of complex numbers. Recall that  $U$  is a \emph{Galois algebra} over $\Gamma$ if $U$ is 
 a Galois ring    over $\Gamma$ and $\k$-algebra. 

Let $V$ be a finite dimensional complex vector space, $S=S(V^*)=$ $\mathbb{C}[x_1,\ldots, x_n]$, and $L = \Frac \, S$. Let $G$ be a unitary reflection group which is a product of groups of  type $G(m,p,n)$. Consider 
the tensor product of polynomial algebras $S \otimes \mathbb{C}[w_1,\ldots,w_m]$, with the trivial action of $G$ on the second component.

A  \emph{quantum linear Galois algebra} $U$ is a Galois algebra over an appropriate $\Gamma$ in $(\mathbb{C}(x_1,\ldots,x_n;w_1,\ldots,w_m)*\mathbb{Z}^n)^G$ or $(\mathbb{C}(x_1,\ldots,x_n;w_1,\ldots,w_m)*\mathbb{N}^n)^G$, where a basis $\epsilon_1, \ldots, \epsilon_n$  of either $\mathbb{Z}^n$ or $\mathbb{N}^n$ acts as follows: $\epsilon_i(x_i) = qx_i$, $ \epsilon_i(x_j) = x_j$, $j \neq i$,  $i,j=1,\ldots, n$.

 Note that the quantum universal enveloping algebra $U_q(gl_n)$ and  the quantum orthogonal Gelfand-Zetlin algebras of type $A$ are examples of quantum linear Galois algebras
 \cite{FH}, \cite{Hartwig}.   The results of the previous sections show that the following algebras are also quantum linear Galois algebras:\\
 
 \
 \begin{itemize}
 \item  $O_q(\k^{2n})^G$ for $G=G(m,p,n)$;
 \item $A_n^{{q}}(\k)^{S_n}$;
 \item $O_q(\k*^{2n})^G$ for $G=G(m,p,n)$.
 \end{itemize}
\

The following theorem shows that the quantum Gelfand-Kirillov Conjecture holds for quantum linear Galois algebras, which is the quantum
 analogue of \cite{Eshmatov}, Theorem 6.

\begin{theorem}\label{thm-qlga}
Let $U$ be a quantum linear Galois algebra in $$(\mathbb{C}(x_1,\ldots,x_n;w_1,\ldots,w_m)*X^n)^G,$$ where $X$ is either $\mathbb{Z}$ or $\mathbb{N}$, with the $G$ action as above. Then  the quantum Gelfand-Kirillov conjecture holds for $U$ and
  there exist $l=(l_1,\ldots, l_n) \in \mathbb{Z}^n$  such that $$\Frac \, U \cong \Frac \, (O_{\overline{q}}(k^{2n}) \otimes \mathbb{C}[w_1,\ldots, w_n]),$$ 
 where $\overline{q}=(q^{l_1},\ldots, q^{l_n})$.
\end{theorem}
 
\begin{proof}
The proof follows from Proposition \ref{prop-supp}, (ii) and the positive solution of the $q$-difference Noether problem for $G$ \cite{Hartwig}.
\end{proof}

\subsection{Skew field of fractions of $U(sl_2)$}

Consider the standard basis $e,f,h$ of $sl_2$, where $[h,e]=e$, $[h,f]=-f$, $[e,f]=2h$. 
The universal enveloping algebra $U(sl_2)$ can be realized as a generalized Weyl algebra $\k[H,C](\sigma, a)$, where $a=C-H(H+1))$,
 with the isomorphism given by $e \mapsto X$, $f \mapsto Y$, $h \mapsto H,$  $h(h+1)+fe \mapsto C$.

Define an action of the cyclic group $G_m$ of order $m$, $m>1$ on $U(sl_2)$ as follows. Denote by $g$ a generator of $G_m$.
Then $g$ fixes $h$ and sends $e \mapsto \xi e$, $f \mapsto \xi^{-1} f$, where $\xi$ is a fixed $m$th primitive root of unity. 

We have that  $\k[H,C](\sigma, a)$ (and hence $U(sl_2)$)
 is birationally equivalent  to $\k[H,C]*\mathbb{Z}$, where again $\mathbb{Z}$ acts by $\sigma$. 
 The action of $G_m$  naturally  extends to $\k[H,C]*\mathbb{Z}$, where the generator $g$ acts on $\mathbb{Z}$ by sending $\overline{y} \mapsto \xi^{y}, y \in \mathbb{Z}$. Therefore $U(sl_2)^{G_m}$ embedds into $(\k[H,C]*\mathbb{Z})^{G_m}$.  
 Since $C$ is fixed by $\sigma$ and also by the action of $G_m$, we have 
 $$\Frac  (\k[H,C]*\mathbb{Z})^{G_m} \cong \Frac  (\k[C] \otimes (\k[H]*\mathbb{Z})^{G_m}).$$ 
 
 On the other hand, $k[H]*\mathbb{Z}$ is isomorphic to the localization $A_1(\k)_x = A_1(\k)_{x^m}$ (\cite{FO1}, section 7) of the first Weyl algebra. Hence,
  $$\Frac (\k[H]*\mathbb{Z})^{G_m} \cong \Frac  (A_1(\k)_{x^m})^{G_m} \cong \Frac  (A_1(\k)^{G_m}_{x^m}) \cong \Frac (A_1(\k)^{G_m}),$$

where the action of the generator $g$ on  $A_1(\k)$ is as follows: $x \mapsto \xi^{-1} x$,  $\partial \mapsto \xi \partial$. 

We conclude that  $U(sl_2)^{G_m}$ is birationally equivalent to $\k[C] \otimes A_1(\k)^{G_m}$. Taking into account the result of \cite{Alev1}, which implies that  
$A_1(\k)^{G_m}\simeq A_1(\k)$  we finally have

\begin{corollary}\label{cor-sl2}
For any $m>1$ and the action of $G_m$ described above,  we have $$\Frac (U(sl_2)^{G_m}) \cong \Frac (\k[C] \otimes A_1(\k)) \cong \Frac (\k[C] )\otimes \Frac (U(sl_2)).$$
\end{corollary}

The last isomorphism is just the classical Gelfand-Kirillov conjecture for  $sl_2$ \cite{Gelfand}.

Recall, that $U(sl_2)$ is rigid by  \cite{Alev3}, that is there is no non trivial  finite group $G \subset Aut_\k \, U(sl_2)$ such that $U(sl_2)^G \cong U(sl_2)$. By Corollary \ref{cor-sl2},
in spite of the rigidity of $U(sl_2)$ we have $\Frac(U(sl_2)^{G_m})\simeq \Frac(U(sl_2))$, giving an example to the question posed in 
 in \cite{Kirkman}.

\end{document}